\documentclass[11pt]{article}
\usepackage{CJK}
\usepackage[usenames]{color}
\usepackage{extarrows}
\usepackage{lineno,hyperref}
\modulolinenumbers[5]

\usepackage{amsfonts}
\usepackage{amsmath, cases}
\usepackage{indentfirst}
\usepackage{graphicx}
\usepackage{latexsym,bm, amsthm}

\usepackage{epsfig}
\usepackage{latexsym}
\usepackage{amssymb}

\numberwithin{equation}{section}  
\newtheorem{theorem}{Theorem}[section]
\newtheorem{corollary}[theorem]{Corollary}
\newtheorem*{corollary*}{Corollary}

\newtheorem*{claim*}{Claim}
\newtheorem{lemma}[theorem]{Lemma}
\newtheorem*{lemma*}{Lemma}
\newtheorem{proposition}[theorem]{Proposition}
\newtheorem*{proposition*}{Proposition}
\newtheorem{remark}[theorem]{Remark}
\newtheorem*{remark*}{Remark}
\newtheorem*{example*}{Example}

\newtheorem*{question*}{Question}

\newtheorem*{definition*}{Definition}

\begin{document}
\begin{CJK*}{GBK}{song}

\begin{center}
{\large  \bf The $(b, c)$-inverse in semigroups and rings with involution}\\
\vspace{0.8cm} {\small \bf Xiaofeng Chen, Jianlong Chen} \footnote{
E-mail: xfc189130@163.com. Corresponding author, E-mail:
jlchen@seu.edu.cn}

\vspace{0.6cm} {\rm School of Mathematics, Southeast University, Nanjing 210096, China}

\end{center}

\bigskip
{ \bf  Abstract:}  \leftskip0truemm\rightskip0truemm  In this paper, we  first prove  that if $a$ is both left $(b, c)$-invertible and left $(c, b)$-invertible,  then $a$ is both  $(b, c)$-invertible and  $(c, b)$-invertible in a $*$-monoid, which generalized the recent result about the inverse along an element by Wang and Mosi\'{c},  under the conditions $(ab)^{*}=ab$ and $(ac)^{*}=ac$.  In addition, we consider that $ba$ is  $(c,b)$-invertible, and at the same time $ca$ is  $(b,c)$-invertible under the same conditions, which extend the related results about Moore-Penrose inverses by Chen et al. to $(b,c)$-inverses.  As applications, we obtain that under condition $(a^{2})^{*}=a^{2}$, $a$ is an EP element if and only if $a$ is one-sided core invertible if and only if $a$ is group invertible.


 \bigskip

\section { \bf Introduction}
In 2010, Baksalary and Trenkler \cite{BT} introduced the core  inverse of a complex matrix. In 2014, Rak\'{\i}c et al. \cite{Raki} generalized core inverses to the case of an element in a ring with involution, and they use five equations to characterize the core inverse. Later Xu et al. \cite{Xu} proved that these five equations can be reduced to three equations. In fact, the concept of the core inverse characterized by the three equations also holds in a $*$-monoid. In the same way, the Moore--Penrose inverse can also be defined in a $*$-monoid by four equations. In 2011, Mary \cite{Mary} introduced a new generalized inverse, which called the inverse along an element in semigroups. In 2017, Zhu et al. \cite{zhu}  introduced left (right) inverse along an element, which generalized the inverse along an element to one-sided cases. In 2012, Drazin \cite{Dc} introduced concept of $(b,c)$-inverses in semigroups. This generalized inverse is a generalization of some well-known generalized inverses such as the Moore--Penrose inverse, the group inverse, the core inverse and  the inverse along an element. In 2016, Drazin  generalized the $(b,c)$-inverse to  one-sided cases in \cite{Df}.

Let $S$ be a $*$-monoid and $a\in S$. It is well known that $a$ is Moore--Penrose invertible if and only if $a\in aa^{*}S\cap Sa^{*}a$. In 2016, Zhu et al. \cite{zhu} generalized this result to one-sided cases, they proved that $a$ is Moore--Penrose invertible if and only if $a\in aa^{*}aS$ if and only if $a\in Saa^{*}a$. In other words, $a$ is Moore--Penrose invertible if and only if $a$ is left invertible along $a^{*}$ if and only if $a$ is right invertible along $a^{*}$. Recently, Wang and  Mosi\'{c} \cite{wang} extended the above result and showed that if $a$ is  left invertibility along $d$, then $a$ is  invertibility along $d$ under the condition $(ad)^{*}=ad$. Motivated by above results,  we consider when one-sided $(b,c)$-inverses are $(b,c)$-inverses, which generalize above results to the cases of $(b,c)$-inverses in Section 3.

In 2017, Chen et al. \cite{chen} showed that $a$ is Moore--Penrose invertible and $aS=a^{2}S$ if and only if $a^{*}a$ is left invertible along $a^{*}$ in
$S$.  Chen et al. also proved that $a$ is both Moore-Penrose invertible and  group invertible if and only if $a^{*}a$ is invertible along $a$. Motivated by this, we extend this result to the cases of $(b,c)$-inverses under some conditions in Section 4.

Moreover,  Chen et al. \cite{chen} also showed that $a$ is both Moore--Penrose invertible and  group invertible if and only if  $(a^{*}a)^{k}$ is invertible along $a$ if and only if $(aa^{*})^{k}$ is invertible along $a$ in $S$, where $k$ is an arbitrary given positive integer.  Motivated by this, we generalize the above result to the inverse along an element under some conditions in Section 5.

\section { \bf Preliminaries}
Throughout this paper,   $S$ is a $*$-monoid and $R$ is a  unital $*$-ring. we call that $S$ is a $*$-monoid, that is a monoid with involution $*$.  An involution $ *: S\rightarrow S$ is an anti-isomorphism which satisfies $(a^{*})^{*}=a, ~ (ab)^{*}=b^{*}a^{*},$ for all $a, b\in S$.  $R$ is a unital $*$-ring if there exists an  involution $ *: R\rightarrow R$ which is an anti-isomorphism  satisfying $(a^{*})^{*}=a, ~(a+b)^{*}=a^{*}+b^{*}, ~ (ab)^{*}=b^{*}a^{*},$ for all $a, b\in R$.

An element $a\in S$ is said to be Moore--Penrose invertible, if there exists  $x\in S$ satisfying:
$$(1)~ axa=a, ~ (2)~ xax=x, ~~(3) ~(ax)^{*}=ax,  ~(4)~ (xa)^{*}=xa.$$
Any element $x$ satisfying (1)--(4) is called Moore--Penrose inverse of $a$. If such an $x$ exists, it is unique and denoted by $a^{\dag}$. If $x$ satisfies (1), then $x$ is  called an inner inverse of $a$, denoted by $a^{-}$. If $x$ satisfies (1) and (3), then $x$ is a $\{1,3\}$-inverse of $a$, denoted by $a^{(1,3)}$. If $x$ satisfies (1) and (4), then $x$ is a $\{1,4\}$-inverse of $a$, denoted by $a^{(1,4)}$.  An element $a\in S$ is group invertible if there is
an $x$ satisfies (1), (2) and commutes with $a$. The group inverse of $a$ is unique if it exists and is denoted by $a^{\#}$. The sets of all Moore--Penrose invertible elements,  $\{1, 3\}$-invertible elements, $\{1, 4\}$-invertible elements and group invertible elements in $S$ are denoted by the symbols $S^{\dag}$, $S^{(1,3)}$, $S^{(1,4)}$ and $S^{\#}$, respectively.

An element $a\in S$ is said to be core invertible \cite{Xu}, if there exists  $x\in S$ satisfying:
$$ ax^{2}=x, ~~  xa^{2}=a, ~~(ax)^{*}=ax.$$
such an element $x$ is unique if it exists and called the core inverse of $a$,  denoted by $a^{\tiny\textcircled{\tiny\#}}$. The dual
concept of core inverse which is called dual core inverse. The sets of all the core invertible and dual core invertible elements in $S$ are denoted by $S^{\tiny\textcircled{\tiny\#}}$ and $S_{\tiny\textcircled{\tiny\#}}$, respectively.

If $a\in S$ is both Moore--Penrose invertible and group invertible with $a^{\dag}$ = $a^{\#}$, then $a$ is said to be an EP element. As we all know, if
$a$ is EP, then $a\in S^{\dag}\cap S^{\#}\cap S^{\tiny\textcircled{\tiny\#}}\cap S_{\tiny\textcircled{\tiny\#}}$ with $a^{\dag}$ = $a^{\#}=a^{\tiny\textcircled{\tiny\#}}=a_{\tiny\textcircled{\tiny\#}}$.

In 2011, Mary \cite{Mary} introduced a new generalized inverse using Green's pre-orders and relations. An element $a\in S$ will be said to be invertible along $d\in S$ if there exists $b\in S$ such that
$$bad=d=dab, ~~bS\subseteq dS, ~~Sb\subseteq Sd.$$
If such $b$ exists, then it is unique and called the inverse  of $a$ along $d$,  denoted by $a^{||d}$.

In \cite{zhu}, Zhu et al. introduced left (right) inverse along an element. An element $a\in S$ is left (resp. right) invertible along $d\in S$ if there exists $b\in S$ such that
$$bad=d~ (\text{resp}. ~dab=d), ~~Sb\subseteq Sd ~(\text{resp}. ~bS\subseteq dS).$$

In 2012, Drazin \cite{Dc} introduced a class of outer  inverses. For any $a, b, c\in S$,  $a\in S$ is said to be $(b, c)$-invertible if there exists $y\in S$ such that $$y\in bSy\cap ySc,~ yab=b,~ cay=c.$$  If such a $y$ exists, it is unique and called  a $(b, c)$-inverse of $a$, denoted by $a^{||(b, c)}$. Moreover, $a^{\dag}=a^{||(a^{*}, a^{*})}$, $a^{\#}=a^{||(a, a)}$, $a^{\tiny\textcircled{\tiny\#}}=a^{||(a, a^{*})}$ and $a^{||d}=a^{||(d, d)}$.

In 2016, Drazin \cite{Df} generalized the $(b,c)$-inverse to one-sided cases: For $a, b, c\in S$,  $a$ is left $(b, c)$-invertible if $b\in Scab$, or equivalently if there exists $x\in Sc$ such that $xab=b$, in which case any such $x$ will be called a left $(b, c)$-inverse of $a$. Dually, $a$ is called right $(b, c)$-invertible if $c\in cabS$, or equivalently if there exists $z\in bS$ such that $caz=c$, in which case any such $z$ will be called a right $(b, c)$-inverse of $a$.

\section { \bf When $a$ is both $(b,c)$-invertible and $(c,b)$-invertible?}
 In this section, we prove that if $a$ is both left $(b, c)$-invertible and left $(c, b)$-invertible, then $a$ is both $(b, c)$-invertible and  $(c, b)$-invertible under the conditions $(ab)^{*}=ab$ and $(ac)^{*}=ac$ in $S$. First, we will give several auxiliary lemmas.
\begin{lemma}\label{lem11} Let $a, b, c\in S$.
\begin{itemize}
\item[\rm(1)] \rm{\cite[Theorem 2.2]{Dc}}  $a$ is $(b, c)$-invertible if and only if $b\in Scab$ and $c\in  cabS$.

\item[\rm(2)] \rm{\cite[Theorem 2.1]{Df}} $a$ is  $(b, c)$-invertible with $(b,c)$-inverse $y$ if and only if $a$ is left $(b, c)$-invertible  with a left $(b, c)$ inverse $x$ and $a$ is right $(b, c)$-invertible  with a right $(b, c)$ inverse $z$. In this case $y=x=z$.
\end{itemize}
\end{lemma}

\begin{lemma}\rm{\cite[Theorem 2.1(ii) and Proposition 6.1]{Dc}}\label{lem10} Let $a, b, c\in S$.  Then
$y\in S$ is the $(b, c)$-inverse of $a$ if and only if $yay = y,~ yS = bS$ and $Sy = Sc$.
\end{lemma}

\begin{lemma}\emph{\cite{Hart}} \label{lem6}
Let $a, x\in S$. Then
\begin{itemize}
\item[\rm(1)] $x$ is a $\{1,3\}$-inverse of $a$ if and only if $a=x^{*}a^{*}a$;

\item[\rm(2)]$x$ is a $\{1,4\}$-inverse of $a$ if and only if $a=aa^{*}x^{*}$.
\end{itemize}
\end{lemma}

\begin{lemma}\emph{\cite{Ko}}\label{lem8}
An element $a\in S$ is EP if and only if $a$ is group invertible and $aa^{\#}=(aa^{\#})^{*}$.
\end{lemma}

\begin{lemma}\emph{\cite{Raki}} \label{lem9}Let $a\in S$. Then the following conditions are equivalent:
\begin{itemize}
\item[\rm(1)] $a\in S^{\dag}\cap S^{\#}$;

\item[\rm(2)] $a\in S^{\tiny\textcircled{\tiny\#}}\cap S_{\tiny\textcircled{\tiny\#}}$.

\end{itemize}
\end{lemma}

The next theorem is the main result of this section, which generalized  the result of  \cite[Theorem 3.8]{wang} to $(b,c)$-inverses.
\begin{theorem} \label{the1}Let $a, b, c\in S$. If $(ab)^{*}=ab$, $(ac)^{*}=ac$. We have the following results:
\begin{itemize}
\item[\rm(I)] The following conditions are equivalent:
\begin{itemize}
\item[\rm(1)] $a$ is both left $(b, c)$-invertible and left $(c, b)$-invertible;

\item[\rm(2)] $a$ is both $(b, c)$-invertible and  $(c, b)$-invertible.
\end{itemize}
In this case, $a^{||(b, c)}=xc$ and $a^{||(c, b)}=yb$, where $b=xcab$ and $c=ybac$.
\item[\rm(II)] The following conditions are equivalent:
\begin{itemize}
\item[\rm(1)] $a$ is both right $(b, c)$-invertible and right $(c, b)$-invertible;

\item[\rm(2)] $a$ is  both $(b, c)$-invertible and  $(c, b)$-invertible.
\end{itemize}
In this case, $a^{||(b, c)}=by$ and $a^{||(c, b)}=cx$, where $b=bacx$ and $c=caby$.
\end{itemize}
\end{theorem}
\begin{proof}
(I) (1)$\Rightarrow$(2). Since $a$  is both left $(b, c)$-invertible and left $(c, b)$-invertible, there exist $x, y\in S$ such that $b=xcab$ and $c=ybac$. From
$(ab)^{*}=ab$, $(ac)^{*}=ac$ and
\begin{equation*}
bc^{*}x^{*}=xcabc^{*}x^{*}=(xc)ab(xc)^{*},
\end{equation*}
\begin{equation*}
cb^{*}y^{*}=ybacb^{*}y^{*}=(yb)ac(yb)^{*},
\end{equation*}
we have
\begin{equation}\label{1}
bc^{*}x^{*}=(bc^{*}x^{*})^{*}=xcb^{*},
\end{equation}
\begin{equation}\label{2}
cb^{*}y^{*}=(cb^{*}y^{*})^{*}=ybc^{*}.
\end{equation}
Let $z=xc$, then $b=xcab=zab\in zS$ and
\begin{equation}\label{3}
\begin{split}
z&=xc=xybac=xyb(ac)^{*}=xybc^{*}a^{*}\xlongequal[]{(\ref{2})}xcb^{*}y^{*}a^{*}\\
&\xlongequal[]{(\ref{1})}bc^{*}x^{*}y^{*}a^{*}
=b(ayxc)^{*}\in bS.
\end{split}
\end{equation}
We obtain $bS=zS$. Since
\begin{equation*}
\begin{split}
c&=ybac=yb(ac)^{*}=ybc^{*}a^{*}\xlongequal[]{(\ref{2})}cb^{*}y^{*}a^{*}=c(ayb)^{*}\\
&=c(ayxcab)^{*}=c(ab)^{*}(ayxc)^{*}=cab(ayxc)^{*}\xlongequal[]{(\ref{3})}caz,
\end{split}
\end{equation*}
we have $Sc=Sz$. From
$$zaz=xcaz=xc=z,$$
it implies that $a^{||(b, c)}=z=xc$ by Lemma \ref{lem10}. Exchange $b$ and $c$, we have $a^{||(c, b)}=yb$.

(2)$\Rightarrow$(1). It is clear.

(II) (1)$\Rightarrow$(2). From that $a$  is both right $(b, c)$-invertible and right $(c, b)$-invertible, there exist $x, y\in S$ such that $b=bacx$ and $c=caby$. Since\\
$$x^{*}acab=x^{*}acabacx=x^{*}(ac)(ab)(ac)x,$$
thus,
\begin{equation}\label{4}
x^{*}acab=(x^{*}acab)^{*}=(ab)^{*}(ac)^{*}x=abacx.
\end{equation}
Then
\begin{equation*}
\begin{split}
b&=bacx=b(ac)^{*}x=bc^{*}a^{*}x=b(caby)^{*}a^{*}x=by^{*}abc^{*}a^{*}x\\
&=by^{*}abacx\xlongequal[]{(\ref{4})}by^{*}x^{*}acab\in Scab,
\end{split}
\end{equation*}
therefore, $a$ is left $(b, c)$-invertible. Thus, $a$ is $(b, c)$-invertible and $a^{||(b, c)}=by^{*}x^{*}ac=by$ by Lemma \ref{lem11}.  In the similar way, $a$ is  $(c, b)$-invertible and $a^{||(c, b)}=cx$.

(2)$\Rightarrow$(1). It is clear.
\end{proof}

Take $b=c=d$ in Theorem \ref{the1}. Then we have the following corollary, which yields the result of  \cite[Theorem 3.8]{wang}. Moreover, we also prove that $d\in S^{(1,4)}$ and give the expression of $a^{||d}$.

\begin{corollary}\label{cor1} Let $a, d\in S$. If $(ad)^{*}=ad$, then
\begin{itemize}
\item[\rm(1)]
 $a$ is left invertible along $d$ if and only if $a$ is invertible along $d$. In this case,  $d\in S^{(1,4)}$ and $a^{||d}=xd$, where $d=xdad$.

\item[\rm(2)] $a$ is right invertible along $d$ if and only if $a$ is invertible along $d$. In this case,  $d\in S^{(1,4)}$ and  $a^{||d}=dx$, where $d=dadx$.
\end{itemize}
\end{corollary}
\begin{proof}
(1). Let $b=c=d$ in Theorem \ref{the1}(I), we have $a^{||d}=xd$, where $d=xdad$.  From the proof of \cite[Theorem 3.8]{wang}, we can see that $d^{*}=axdd^{*}$. Thus, $d=dd^{*}(ax)^{*}$ and $d\in S^{(1,4)}$ by Lemma \ref{lem6}(2).

(2). Let $b=c=d$ in Theorem \ref{the1}(II), we obtain $a^{||d}=dx$, where $d=dadx$. Then $d=dadx=d(ad)^{*}x=dd^{*}a^{*}x$, thus $d\in S^{(1,4)}$ by Lemma \ref{lem6}(2).
\end{proof}

\begin{remark}
 If $(ad)^{*}=ad$,  $a$ is  invertible along $d$, then  $d\in S^{(1,4)}$ from Corollary \ref{cor1}. But $d\notin S^{(1,3)}$.
For example: let $S=\mathbb{C}^{2\times 2}$ and the involution be the  transpose. Let  $d=\left[\begin{matrix}
1 & 0\\
i & 0
\end {matrix}
\right],$
$a=\left[\begin{matrix}
1 & 0\\
-i & 1
\end {matrix}
\right].$ Then $(ad)^{*}=ad$, $(da)^{*}\neq da$ and $a$ is  invertible along $d$, but $d$ is not $\{1,3\}$-invertible.
\end{remark}

Take $b=a$, $c=a^{*}$ in Theorem \ref{the1}, then we have that $a$ is one-sided core invertible if and only if $a$ is core invertible  under condition $(a^{2})^{*}=a^{2}$. Moreover, we can prove that $a$ is one-sided core invertible if and only if $a$ is  an EP element.

\begin{theorem}\label{the10} Let $a\in S$. If $(a^{2})^{*}=a^{2}$. Then the following  conditions are equivalent:
\begin{itemize}
\item[\rm(1)] $a$ is left core invertible;

\item[\rm(2)] $a$ is right core invertible;

\item[\rm(3)] $a$ is EP.
\end{itemize}
\end{theorem}
\begin{proof}
(1)$\Rightarrow$(3). Since $a$ is left core invertible, there exists $x\in S$ such that $a=xa^{*}a^{2}$. Then from $(a^{2})^{*}=a^{2}$ and
$$a^{2}x^{*}=(xa^{*}a^{2})ax^{*}=(xa^{*})a^{2}(xa^{*})^{*},$$
we have
\begin{equation}\label{5}
 a^{2}x^{*}=(a^{2}x^{*})^{*}=xa^{2}.
\end{equation}
And
\begin{equation}\label{6}
 a^{*}a^{2}=a^{*}(a^{2})^{*}=(a^{2})^{*}a^{*}=a^{2}a^{*}.
\end{equation}
Hence
$$a=xa^{*}a^{2}\xlongequal[]{(\ref{6})}xa^{2}a^{*}\xlongequal[]{(\ref{5})}a^{2}x^{*}a^{*}=a^{2}(ax)^{*}.$$
Then $a=a^{2}(ax)^{*}=a(a^{2}(ax)^{*})(ax)^{*}=a^{3}((ax)^{2})^{*}$. From $$a^{*}=(a^{3}((ax)^{2})^{*})^{*}=(ax)^{2}(a^{2})^{*}a^{*}=(ax)^{2}a^{2}a^{*}\in Sa^{2}a^{*},$$
we have $a$ is left $(a^{*}, a)$-invertible. Hence,  by Theorem \ref{the1}, $a^{\tiny\textcircled{\tiny\#}}$ and $a_{\tiny\textcircled{\tiny\#}}$ exist and
\begin{equation*}
a^{\tiny\textcircled{\tiny\#}}=xa^{*} ~~and ~~a_{\tiny\textcircled{\tiny\#}}=(ax)^{2}a.
\end{equation*}
Since $(aa^{\tiny\textcircled{\tiny\#}})^{*}=aa^{\tiny\textcircled{\tiny\#}}$, we have $axa^{*}=(axa^{*})^{*}=ax^{*}a^{*}$. From Lemma \ref{lem9}, we have $a\in S^{\#}$, and since $a=a^{2}(ax)^{*}$, we obtain $a^{\#}=a((ax)^{*})^{2}$. Thus,
$$aa^{\#}=aa((ax)^{*})^{2}=a^{2}(ax)^{*}(ax)^{*}=a(ax)^{*}=ax^{*}a^{*}=axa^{*}=(ax^{*}a^{*})^{*}=(aa^{\#})^{*}.$$
Then $a$ is EP by Lemma \ref{lem8} and $a^{\#}=a^{\dag}=a^{\tiny\textcircled{\tiny\#}}=a_{\tiny\textcircled{\tiny\#}}=xa^{*}=(ax)^{2}a$.

(2)$\Rightarrow$(3).  Since $a$ is right core invertible, there exists $y\in S$ such that $a^{*}=a^{*}a^{2}y$. Then from $(a^{2})^{*}=a^{2}$ and
$$y^{*}a^{2}=y^{*}a^{*}a^{*}=y^{*}a^{*}a^{*}a^{2}y=y^{*}(a^{*})^{2}a^{2}y,$$
we have
\begin{equation}\label{7}
 y^{*}a^{2}=(y^{*}a^{2})^{*}=a^{2}y.
\end{equation}
Hence
$$a^{*}=a^{*}a^{2}y\xlongequal[]{(\ref{7})}a^{*}y^{*}a^{2}=(ya)^{*}a^{2}.$$
Then
\begin{equation*}
\begin{split}
a^{*}&=(ya)^{*}a^{2}=(ya)^{*}a^{*}a^{*}=(ya)^{*}((ya)^{*}a^{2})a^{*}=((ya)^{*})^{2}a^{*}(a^{*})^{2}\\
&=((ya)^{*})^{2}(ya)^{*}a^{2}(a^{*})^{2}=((ya)^{*})^{3}a^{2}(a^{*})^{2}.\\
\end{split}
\end{equation*}
Thus, $a=a^{2}a^{*}a^{*}(ya)^{3}\in a^{2}a^{*}S$,
we have that $a$ is right $(a^{*}, a)$-invertible. Hence,  by Theorem \ref{the1}, $a^{\tiny\textcircled{\tiny\#}}$ and $a_{\tiny\textcircled{\tiny\#}}$ exist. Since $a^{*}=(ya)^{*}a^{2}=(ya)^{*}(a^{2})^{*}=(a^{2}ya)^{*}$, we have $a=a^{2}ya$. Thus
\begin{equation*}
a^{\tiny\textcircled{\tiny\#}}=ay ~~and ~~a_{\tiny\textcircled{\tiny\#}}=a(ya)^{2}.
\end{equation*}

From Lemma \ref{lem9}, we have $a\in S^{\#}$, and since $a=a^{2}ya$, we obtain $a^{\#}=a(ya)^{2}=a_{\tiny\textcircled{\tiny\#}}$. Thus,
$$a^{\#}a=a_{\tiny\textcircled{\tiny\#}}a=(a_{\tiny\textcircled{\tiny\#}}a)^{*}=(a^{\#}a)^{*}.$$
Then $a$ is EP by Lemma \ref{lem8} and $a^{\#}=a^{\dag}=a^{\tiny\textcircled{\tiny\#}}=a_{\tiny\textcircled{\tiny\#}}=ay=a(ya)^{2}$.

(3)$\Rightarrow$(1) and (2). Since $a$ is EP, we have that $a$ is core invertible, therefore, $a$ is left and right core invertible.
\end{proof}

\section { \bf When $ba$ is  $(c,b)$-invertible and $ca$ is $(b,c)$-invertible?}
In this section,  we give some equivalent conditions that $ba$ is   $(c, b)$-invertible and  $ca$ is $(b, c)$-invertible under conditions $(ab)^{*}=ab$ and $(ac)^{*}=ac$ in $S$.  In \cite[Theorem 3.1]{chen}, Chen et al. proved that $a\in S^{\dag}$ and $aS=a^{2}S$ if and only if $a^{*}\in S^{\dag}$  and $Sa^{*}=S(a^{2})^{*}$ if and only if $a^{*}a$ is left invertible along $a^{*}$. We first extend this result to $(b,c)$-inverses in the following theorem.
\begin{theorem} \label{the2}Let $a, b, c\in S$. If $(ab)^{*}=ab$, $(ac)^{*}=ac$. Then the following conditions are equivalent:
\begin{itemize}
\item[\rm(1)] $a$ is both $(b, c)$-invertible and $(c, b)$-invertible, $Sb=Sb^{2}, $ $Sc=Sc^{2}$;

\item[\rm(2)] $ba$ is  left $(c, b)$-invertible and  $ca$ is left $(b, c)$-invertible.
\end{itemize}
\end{theorem}
\begin{proof}
(1)$\Rightarrow$(2). Since  $a$ is  both $(b, c)$-invertible and $(c, b)$-invertible, we have \\
$$Sb=Scab ~~~\text{and}~~~ Sc=Sbac.$$
Combining with $Sb=Sb^{2},$ $Sc=Sc^{2}$, which gives that \\
$$Sb=Sc^{2}ab ~~~\text{and}~~~ Sc=Sb^{2}ac.$$
Therefore, $ba$ is  left $(c, b)$-invertible and  $ca$ is left $(b, c)$-invertible.

(2)$\Rightarrow$(1). From $Sb=Sc^{2}ab\subseteq Scab$,  $Sc=Sb^{2}ac\subseteq Sbac$ and Theorem \ref{the1}, we have that $a$ is  both $(b, c)$-invertible and $(c, b)$-invertible. And there exist $t, s\in S$ such that\\
$$b=tc^{2}ab=tc(cab)~~~ \text{and}~~~ c=sb^{2}ac=sb(bac).$$
By the proof of Theorem \ref{the1}(I), we obtain\\
$$Sc\subseteq Stc^{2}\subseteq Sc^{2}\subseteq Sc,$$
thus, $Sc=Sc^{2}$. In the similar way, we have $Sb=Sb^{2}$.
\end{proof}

Let $b=c=d$ in Theorem \ref{the2}, we obtain the following corollary.
\begin{corollary}\label{cor3}
Let $a, d\in S$, $(ad)^{*}=ad$. Then the following  conditions are equivalent:
\begin{itemize}
\item[\rm(1)] $a^{||d}$ exists,  $Sd=Sd^{2}$;

\item[\rm(2)] $da$ is left invertible along $d$.
\end{itemize}
\end{corollary}

\begin{remark}
In Corollary  \ref{cor3}, the condition $Sd=Sd^{2}$ cannot be deleted.
For example: let $S=\mathbb{C}^{2\times 2}$ and the involution be the conjugate transpose.  Take $a=\left[\begin{matrix}
1 & i\\
i & -1
\end {matrix}
\right],$ $d=a^{*}$. Since $a\in S^{\dag}$, it gives that $a^{||d}$ exists. From $d^{2}ad=0$, we have $d\notin Sd^{2}ad$, thus $da$ is not left invertible along $d$.
\end{remark}

 Since $Sa=Sa^{2}$ if and only if  $a^{*}a$ is left invertible along $a$ under the condition $a\in S^{\dag}$ in \cite{chen},  we next generalize this result to the cases of $(b,c)$-inverses in the following result.
\begin{proposition}\label{the3}Let $a, b, c\in S$.
\begin{itemize}
\item[\rm(1)] If $a$ is both left $(b, c)$-invertible and left $(c, b)$-invertible, then\\
$Sb=Sb^{2}, $ $Sc=Sc^{2}$ if and only if $ab$ is  left $(b, c)$-invertible and  $ac$ is left $(c, b)$-invertible.

\item[\rm(2)] If $a$ is both right $(b, c)$-invertible and right $(c, b)$-invertible, then\\
 $bS=b^{2}S,$ $cS=c^{2}S$ if and only if $ba$ is  right $(c, b)$-invertible and  $ca$ is right $(b, c)$-invertible.
\end{itemize}
\end{proposition}
\begin{proof}
(1).  Since $a$ is both left $(b, c)$-invertible and left $(c, b)$-invertible, there exist $x, y\in S$ such that $b=xcab$ and $c=ybac$.

From $Sb=Sb^{2}, $ $Sc=Sc^{2}$, there exist $t_{1}, t_{2}\in S$ such that $b=t_{1}b^{2}$ and $c=t_{2}c^{2}$. Then
$$b=t_{1}b^{2}=t_{1}(xcab)b=t_{1}xcab^{2}\in Scab^{2},$$
$$c=t_{2}c^{2}=t_{2}(ybac)c=t_{2}ybac^{2}\in Sbac^{2}.$$
Therefore, $ab$ is  left $(b, c)$-invertible and  $ac$ is left $(c, b)$-invertible.

Conversely,  Since $ab$ is  left $(b, c)$-invertible and  $ac$ is left $(c, b)$-invertible, we have \\
$$Sb=Scab^{2}\subseteq Sb^{2},$$
$$Sc=Sbac^{2}\subseteq Sc^{2}.$$
The proof is completed.

(2). It is similar to (1).
\end{proof}

Take $b=c=d$ in Proposition \ref{the3}, then the following result is valid for any integer $k\geq 1$.
\begin{corollary}\label{th4}
Let $a, d\in S$ with $k\geq1$.
\begin{itemize}
\item[\rm(1)] If $a$ is left invertible along $d$. Then\\
~~~$Sd=Sd^{2}$ if and only if $(ad)^{k}$ is left invertible along $d$.

\item[\rm(2)] If $a$ is right invertible along $d$. Then\\
~~~$dS=d^{2}S$ if and only if $(da)^{k}$ is right invertible along $d$.
\end{itemize}
\end{corollary}
\begin{proof}
(1). Suppose that $a$ is left invertible along $d$, there exists $t\in S$ such that $d=tdad$.

Since $Sd=Sd^{2}$, there exists $s\in S$ such that $d=sd^{2}$. Then $d=sd^{2}=s(tdad)d=st(tdad)ad^{2}=st^{2}d(ad)^{2}d=\cdots=st^{k}d(ad)^{k}d\in Sd(ad)^{k}d$. Hence, $(ad)^{k}$ is left invertible along $d$.

Conversely, $(ad)^{k}$ is left invertible along $d$ implies that  $Sd=Sd(ad)^{k}d=S(da)^{k}d^{2}\subseteq Sd^{2}$.

(2). It is similar to (1).
\end{proof}

Take $d=a^{*}$ in above result, then we have the following corollary.
\begin{corollary}\emph{\cite[Theorem 3.6]{chen}}
Let $a\in S^{\dag}$, $k\geq1$. Then
\begin{itemize}
\item[\rm(1)] $Sa=Sa^{2}$ if and only if  $(a^{*}a)^{k}$ is left invertible along $a$;

\item[\rm(2)] $aS=a^{2}S$ if and only if $(aa^{*})^{k}$ is right invertible along $a$.

\end{itemize}
\end{corollary}

In Theorem \ref{the2}, we investigate equivalent conditions that $ba$ is  left $(c, b)$-invertible and  $ca$ is left $(b, c)$-invertible under conditions $(ab)^{*}=ab$, $(ac)^{*}=ac$. Next, we consider that $ba$ is   $(c, b)$-invertible and  $ca$ is  $(b, c)$-invertible under the same conditions.
\begin{theorem} \label{the8}Let $a, b, c\in S$. If $(ab)^{*}=ab$, $(ac)^{*}=ac$. Then the following conditions are equivalent:
\begin{itemize}
\item[\rm(1)] $a$ is  both $(b, c)$-invertible and $(c, b)$-invertible, $b\in S^{\#}$, $c\in S^{\#}$;

\item[\rm(2)] $ba$ is   $(c, b)$-invertible and  $ca$ is  $(b, c)$-invertible.
\end{itemize}
\end{theorem}
\begin{proof}
(1)$\Rightarrow$(2). From Theorem \ref{the2}, we have that $ba$ is left  $(c, b)$-invertible and  $ca$ is left $(b, c)$-invertible.
 $ba$ is right  $(c, b)$-invertible and  $ca$ is right $(b, c)$-invertible by Proposition \ref{the3}(2). Thus, $ba$ is   $(c, b)$-invertible and  $ca$ is  $(b, c)$-invertible.

(2)$\Rightarrow$(1). By Theorem \ref{the2}, we have that  $a$ is both $(b, c)$-invertible and $(c, b)$-invertible, $Sb=Sb^{2}, $ $Sc=Sc^{2}$. From Proposition \ref{the3}(2),
we obtain $bS=b^{2}S,$ $cS=c^{2}S$. Thus $b\in S^{\#}$ and $c\in S^{\#}$.

\end{proof}

Take $b=c=d$ in Theorem \ref{the8}, we have
\begin{corollary}\label{cor4}
Let $a, d\in S$, $(ad)^{*}=ad$. Then the following conditions are equivalent:
\begin{itemize}
\item[\rm(1)] $a^{||d}$ exists,  $d\in S^{\#}$;

\item[\rm(2)] $da$ is  invertible along $d$.
\end{itemize}
\end{corollary}

Take $b=a,~ c=a^{*}$ in Theorem \ref{the2} and Theorem \ref{the8}. Then  by Theorem \ref{the10}, we have
\begin{corollary}
Let $a\in S$, $(a^{2})^{*}=a^{2}$. Then the following conditions are equivalent:
\begin{itemize}
\item[\rm(1)] $a$ is EP;

\item[\rm(2)] $a^{2}$  is left $(a^{*}, a)$-invertible and $a^{*}a$ is left $(a, a^{*})$-invertible;

\item[\rm(3)] $a^{2}$  is  $(a^{*}, a)$-invertible and $a^{*}a$ is  $(a, a^{*})$-invertible.

\end{itemize}
\end{corollary}

\section { \bf When $(ad)^{k}$ is invertible along $a$ and $(da)^{k}$ is  invertible along $a$?}
In this section, we consider that $(ad)^{k}$ is invertible along $a$ and $(da)^{k}$ is  invertible along $a$ under conditions $(ad)^{*}=ad$ and  $(da)^{*}=da$.
We start this section with auxiliary lemmas.
\begin{lemma}\emph{\cite[Corollary 2.3]{chen}}\label{lem4}
Let $a\in R$, $k\geq1$, and let $m\in R$ be regular with inner inverse $m^{-}$. Then the following conditions are equivalent:
\begin{itemize}
\item[\rm(1)] $a$ is left invertible along $m$;

\item[\rm(2)] $u=(am)^{k}+1-m^{-}m$ is left invertible;

\item[\rm(3)] $v=(ma)^{k}+1-mm^{-}$ is left invertible.

\end{itemize}
\end{lemma}

\begin{lemma}\emph{\cite[Corollary 2.7]{chen}}\label{lem2}
Let $a\in R$, $k\geq1$, and let $m\in R$ be regular with inner inverse $m^{-}$. Then the following conditions are equivalent:
\begin{itemize}
\item[\rm(1)] $a$ is right invertible along $m$;

\item[\rm(2)] $u=(am)^{k}+1-m^{-}m$ is right invertible;

\item[\rm(3)] $v=(ma)^{k}+1-mm^{-}$ is right invertible.

\end{itemize}
\end{lemma}

 In \cite[Theorem 3.1]{chen}, Chen et al. proved that for any  integer $k\geq1$, $a\in S^{\dag}$ and $aS=a^{2}S$ if and only if $a^{*}\in S^{\dag}$  and $Sa^{*}=S(a^{2})^{*}$ if and only if $(a^{*}a)^{k}$ is left invertible along $a^{*}$. In Corollary \ref{cor3}, we have extended it to the cases of the inverse along an element when $k=1$. We have no idea whether it holds for  $k\geq1$ in a $*$-monoid. But we will prove that it is valid in $*$-ring for any integer $k\geq1$ in the following theorem.

\begin{theorem}\label{th8}
Let $a, d\in R$, $(ad)^{*}=ad$, $k\geq1$. Then the following conditions are equivalent:
\begin{itemize}
\item[\rm(1)] $a^{||d}$ exists,  $Rd=Rd^{2}$;

\item[\rm(2)] $(da)^{k}$ is left invertible along $d$.
\end{itemize}
\end{theorem}
\begin{proof}
(1)$\Rightarrow$(2). Since $a$ is  invertible along $d$, there exists $t\in R$ such that $d=tdad$. Then, we have $d=tdad=t(tdad)ad=t^{2}(da)^{2}d=\cdots=t^{k}(da)^{k}d$. According to $Rd=Rd^{2}$, there exists $s\in R$ such that $d=sd^{2}$. Then, $d=t^{k}(da)^{k}d=
t^{k}da(da)^{k-1}d=t^{k}sd^{2}a(da)^{k-1}d=t^{k}sd(da)^{k}d\in Rd(da)^{k}d$. Therefore, $(da)^{k}$ is left invertible along $d$.

(2)$\Rightarrow$(1). Suppose that $(da)^{k}$ is left invertible along $d$, there exists $t\in R$ such that $d=td(da)^{k}d$. Thus, $d=td(da)^{k}d=td(da)^{k-1}dad\in Rdad$, which implies that $a$ is  invertible along $d$ by  Corollary \ref{cor1}. Thus $d$ is regular. By Lemma \ref{lem4} and Lemma \ref{lem2}, we have $d(da)^{k}+1-dd^{-}$ is left invertible and $(da)^{k}+1-dd^{-}$ is invertible. From
$(d^{2}d^{-}+1-dd^{-})((da)^{k}+1-dd^{-})=d(da)^{k}+1-dd^{-}$, we obtain that $d^{2}d^{-}+1-dd^{-}$ is left invertible, thus, $d+1-dd^{-}$ is left invertible by Jacobson Lemma. Then there exists $r\in R$ such that $r(d+1-dd^{-})=1$. Hence, $d=r(d+1-dd^{-})d=rd^{2}\in Rd^{2}$.
\end{proof}

Since  in \cite[Theorem 3.4]{chen} Chen et al.  proved that $a\in S^{\dag}$ and $Sa=Sa^{2}$ if and only if $a^{*}\in S^{\dag}$  and $a^{*}S=(a^{2})^{*}S$ if and only if $(aa^{*})^{k}$ is right invertible along $a^{*}$. We extend it to the inverse along an element under condition $(ad)^{*}=ad$ in the following theorem.
\begin{theorem}\label{th2}
Let $a, d\in S$, $(ad)^{*}=ad$, $k\geq1$. Then the following conditions are equivalent:
\begin{itemize}
\item[\rm(1)] $a^{||d}$ exists,  $dS=d^{2}S$;

\item[\rm(2)] $(ad)^{k}$ is right invertible along $d$.

\end{itemize}
\end{theorem}
\begin{proof}
(1)$\Rightarrow$(2). Since $a$ is  invertible along $d$, there exists $h\in S$ such that $d=dadh$. Then, we have $d=dadh=da(dadh)h=d(ad)^{2}h^{2}=\cdots=d(ad)^{k}h^{k}$. According to $dS=d^{2}S$, there exists $s\in S$ such that $d=d^{2}s$. Then, $d=d(ad)^{k}h^{k}=
d(ad)^{k-1}adh^{k}=d(ad)^{k-1}ad^{2}sh^{k}=d(ad)^{k}dsh^{k}\in d(ad)^{k}dS$. Therefore, $(ad)^{k}$ is right invertible along $d$.

(2)$\Rightarrow$(1). Suppose that $(ad)^{k}$ is right invertible along $d$, there exists $t\in S$ such that $d=d(ad)^{k}dt$. Thus, $d=d(ad)^{k}dt=dad(ad)^{k-1}dt\in dadS$, which implies that $a$ is  invertible along $d$ by  Corollary \ref{cor1}. Since  $d^{*}=t^{*}d^{*}(ad)^{k}d^{*}$,
we have
$$(ad)^{k}dt=ad(ad)^{k-1}dt=d^{*}a^{*}(ad)^{k-1}dt=t^{*}d^{*}(ad)^{k}d^{*}a^{*}(ad)^{k-1}dt=t^{*}d^{*}(ad)^{2k}dt.$$
It implies that $(ad)^{k}dt=((ad)^{k}dt)^{*}$. Since

~~~~~~~~$d(d^{2}t)^{*}=dt^{*}d^{*}d^{*}=dt^{*}d^{*}t^{*}d^{*}(ad)^{k}d^{*}=d(tdt)^{*}d^{*}(ad)^{k-1}add^{*}\\
~~~~~~~~~~~~~~~~~~~~~~~=d(tdt)^{*}d^{*}(ad)^{k-1}ad(ad)^{k}dtd^{*}=d(tdt)^{*}d^{*}(ad)^{2k-1}addtd^{*}\\
~~~~~~~~~~~~~~~~~~~~~~~=d(tdt)^{*}d^{*}(ad)^{2k-1}ad(ad)^{k}dtdtd^{*}=d(tdt)^{*}d^{*}(ad)^{3k}d(tdt)d^{*}.$\\
It follows that $d(d^{2}t)^{*}=(d(d^{2}t)^{*})^{*}=d^{2}td^{*}$. Thus, $d=d(ad)^{k}dt=d((ad)^{k}dt)^{*}=d(((ad)^{k-1}ad^{2}t)^{*})=d(d^{2}t)^{*}a^{*}(ad)^{k-1}=d^{2}td^{*}a^{*}(ad)^{k-1}\in d^{2}S$. Therefore,
$dS=d^{2}S$.
\end{proof}

Dually, the following result is obtained.
\begin{theorem}\label{th3}
Let $a, d\in S$, $(da)^{*}=da$, $k\geq1$. Then the following conditions are equivalent:
\begin{itemize}
\item[\rm(1)] $a^{||d}$ exists,  $Sd=Sd^{2}$;

\item[\rm(2)] $(da)^{k}$ is left invertible along $d$.

\end{itemize}
\end{theorem}

In \cite[Theorem 3.8]{chen}, Chen et al. showed that $a\in S^{\#}\cap S^{\dag}$ if and only if $(a^{*}a)^{k}$ is invertible along $a$ if and only if $(aa^{*})^{k}$ is invertible along $a$. We generalize this result to the cases of the inverse along an element in the following result.
\begin{theorem}\label{th5}
Let $a, d\in S$ with $k\geq1$, $(ad)^{*}=ad$, $(da)^{*}=da$. Then the following conditions are equivalent:
\begin{itemize}
\item[\rm(1)] $a^{||d}$ exists,  $d\in S^{\#}$;

\item[\rm(2)] $a^{||d}$ exists,  $d\in S^{\#}\cap S^{\dag}$;

\item[\rm(3)] $(ad)^{k}$ is right invertible along $d$, $(da)^{k}$ is left invertible along $d$;

\item[\rm(4)] $(ad)^{k}$ is  invertible along $d$;

\item[\rm(5)] $(da)^{k}$ is  invertible along $d$.
\end{itemize}
\end{theorem}
\begin{proof}
(1)$\Leftrightarrow$(2). From that $a^{||d}$ exists, $(ad)^{*}=ad$, $(da)^{*}=da$ and  Corollary \ref{cor1}, we obtain $d\in  S^{\dag}$.

(1)$\Leftrightarrow$(3). By Theorem \ref{th2} and Theorem \ref{th3}.

(1)$\Rightarrow$(4). Since $a^{||d}$ exists,  $d\in S^{\#}$, $(ad)^{k}$ is left invertible along $d$ by Corollary \ref{th4}(1), and $(ad)^{k}$ is right invertible along $d$ by Theorem \ref{th2}. Thus, $(ad)^{k}$ is  invertible along $d$.

(4)$\Rightarrow$(1). By Theorem \ref{th2}, we have that $a^{||d}$ exists and $dS=d^{2}S$. From Corollary \ref{th4}(1), we get $Sd=Sd^{2}$. Thus $d\in S^{\#}$.

(1)$\Leftrightarrow$(5). Since $a^{||d}$ exists,  $d\in S^{\#}$, we have that $(da)^{k}$ is left invertible along $d$ by Theorem \ref{th3}, from Corollary \ref{th4}(2), we obtain that $(da)^{k}$ is right invertible along $d$. Thus, $(da)^{k}$ is  invertible along $d$.

(5)$\Rightarrow$(1). By Theorem \ref{th3}, we have that  $a^{||d}$ exists and $Sd=Sd^{2}$. From Corollary \ref{th4}(2), we get $dS=d^{2}S$. Thus $d\in S^{\#}$.
\end{proof}

From the proof of Theorem \ref{th5}, we have $a^{||d}$ exists and  $d\in S^{\#}$ if and only if $(ad)^{k}$ is  invertible along $d$ only needs one condition $(ad)^{*}=ad$. In Corollary \ref{cor4}, we have proved that $da$ is invertible along $d$ if and only if $a^{||d}$ exists and $d\in S^{\#}$ under condition $(ad)^{*}=ad$. Next we will show that it is valid when $(da)^{k}$ is invertible along $d$ in $*$-ring $R$.
\begin{proposition}
Let $a, d\in R$, $(ad)^{*}=ad$, $k\geq1$. Then the following conditions are equivalent:
\begin{itemize}
\item[\rm(1)] $a^{||d}$ exists,  $d\in R^{\#}$;

\item[\rm(2)] $(da)^{k}$ is  invertible along $d$.
\end{itemize}
\end{proposition}
\begin{proof}
(1)$\Rightarrow$(2). From Theorem \ref{th8}, we have that $(da)^{k}$ is left invertible along $d$. By Corollary \ref{th4}(2), we obtain that $(da)^{k}$ is right invertible along $d$. Therefore, $(da)^{k}$ is  invertible along $d$.

(2)$\Rightarrow$(1). By Theorem \ref{th8}, we have that $a^{||d}$ exists and $Rd=Rd^{2}$. From Corollary \ref{th4}(2), we obtain $dR=d^{2}R$. Thus $d\in R^{\#}$.
\end{proof}

Take $d=a$ in Theorem \ref{th5}. Then by Lemma \ref{lem9}  and  Theorem \ref{the10}, we have following result.
\begin{corollary}
Let $a\in S$, $(a^{2})^{*}=a^{2}$. Then the following conditions are equivalent:
\begin{itemize}
\item[\rm(1)] $a\in S^{\#}$;

\item[\rm(2)] $a\in S^{\#}\cap S^{\dag}$;

\item[\rm(3)] $a$ is EP.
\end{itemize}
\end{corollary}

\centerline {\bf ACKNOWLEDGMENTS}
This research is supported by the National Natural Science Foundation of China (No. 11771076,  11871145); the Fundamental Research Funds for the Central Universities and the Postgraduate Research $\&$ Practice Innovation Program of Jiangsu Province
(No. KYCX$19\_ 0055$).

\end{CJK*}
\end{document}